\documentclass[11pt,reqno]{amsart}
\usepackage{graphicx,amscd}
\usepackage{amsfonts,amsmath,amssymb}
\usepackage{mathrsfs}
\usepackage{bbm}
\usepackage{enumerate}
\usepackage{hyperref}
\usepackage{accents}

\hypersetup{
    colorlinks=false,
    pdfborder={0 0 0},
	bookmarksnumbered = true,
	bookmarksopen = true
}
\newcommand{\rl}{{\mathbb{R}}}
\newcommand{\cx}{{\mathbb{C}}}

\newcommand{\om}{\omega}

\newcommand{\e}{\varepsilon}

\newcommand {\Sig}{{\Sigma}}

\newtheorem{theorem}{Theorem}[section]
\newtheorem{lemma}[theorem]{Lemma}
\newtheorem{prop}[theorem]{Proposition}
\newtheorem{cor}[theorem]{Corollary}
\newtheorem{conj}{Conjecture}
\theoremstyle{definition}
\newtheorem{defn}[theorem]{Definition}
\newtheorem{rem}[theorem]{Remark}
\newtheorem{ex}[theorem]{Example}

\numberwithin{equation}{section}

\newcommand{\abs}[1]{\left| \hspace{1pt} #1 \hspace{1pt} \right|}

\newcommand{\norm}[1]{\| \hspace{1pt} #1 \hspace{1pt} \|}

\newcommand{\close}[1]{\overline{#1}}

\newcommand{\oball}[2]{\mathbb{B}({#1}\hrs{1},{#2})}

\newcommand{\hrs}[1]{\hspace{#1pt}}
\newcommand{\vrs}[1]{\vspace{#1pt}}

\newcommand{\nat}[1]{\mathbb{N}^{\hrs{0}#1}}
\newcommand{\integ}[1]{\mathbb{Z}^{\hrs{0}#1}}

\newcommand{\real}[1]{\mathbb{R}^{\hrs{0}#1}}
\newcommand{\compx}[1]{\mathbb{C}^{\hrs{0}#1}}

\newcommand{\wh}[1]{\widehat{#1}}

\newcommand{\Chi}{\mathcal{X}}
\newcommand{\pp}[1]{#1_{\mathit{\Delta}}}

\newcommand{\ess}[1]{{#1}^{\mathrm{ess}}}
\newcommand{\trc}[1]{{#1}^{\mathrm{tr}}}

\renewcommand{\it}[1]{\textit{#1}}
\renewcommand{\bf}[1]{\textbf{#1}}

\begin{document}
\title[Open Whitney umbrellas]{Open Whitney umbrellas are locally polynomially convex}
\author{Octavian Mitrea}
\author{Rasul Shafikov}
\address{Department of Mathematics, the University of Western Ontario, London, Ontario, N6A 5B7, Canada}

\begin{abstract}
It is proved that any smooth open Whitney umbrella in $\cx^2$ is locally polynomially convex 
near the singular point.
\end{abstract}

\maketitle

\let\thefootnote\relax
\let\thefootnote\relax\footnote{
R. Shafikov is partially supported by the Natural Sciences and Engineering Research Council of Canada.
}
%%%%%%%%%%%%%%%%% section %%%%%%%%%%%%%%%%%%
\section{Introduction}

The goal of this paper is to give a generalization of a theorem of Shafikov and Sukhov \cite{SS1}, \cite{SS2},
concerning local polynomial convexity of open Whitney umbrellas. Recall that a standard open (or unfolded) Whitney
umbrella is the map
$\pi : \real{2}_{(t,s)} \rightarrow \real{4}_{(x,u,y,v)} \cong \compx{2}_{(z=x+iy, w=u+iv)}$ given by 
\begin{equation} \label{eq:1}
\displaystyle \pi(t,s) = \left( ts, \frac{2t^3}{3}, t^2, s \right).
\end{equation}
The map $\pi$ is a smooth homeomorphism onto its image, nondegenerate except at the origin. It satisfies $\pi^*\omega_{\rm st} = 0$,
where  $\om_{\rm st} = dx \wedge dy + du \wedge dv$ is the standard symplectic form on $\cx^2$, hence $\Sigma := \pi(\real{2})$ is a Lagrangian embedding in 
$\compx{2}$, with an isolated singular point at the origin. If $\phi: \cx^2 \to \cx^2$ is a local symplectomorphism, which
we may assume, without loss of generality, to preserve the origin, then the image $\phi(\Sig)$ is called an open Whitney umbrella. 
It is called locally polynomially convex at the origin if there exists a basis of compact neighbourhoods of the 
origin in $\phi(\Sig)$ that are polynomially convex (see the next section for details). Our main result is the following:

\begin{theorem}\label{t.main}
Let $\phi: \cx^2 \to \cx^2$ be an arbitrary smooth symplectomorphism. Then the surface
$\phi (\Sigma)$ is locally polynomially convex at the origin.
\end{theorem}

This result was proved for a generic real-analytic $\phi$  in \cite{SS1} and for a generic smooth $\phi$ in \cite{SS2}. 
Our theorem establishes polynomial convexity in \it{full generality} in this context. One immediate application of our 
main result is the following.

\begin{cor}\label{c:1}
For $\Sigma$ and $\phi$ as in Theorem~\ref{t.main}, there exists $\e>0$ sufficiently small, such that any continuous function on 
$\phi(\Sigma)\cap \mathbb B(\phi(0),\e)$ can be uniformly approximated by holomorphic polynomials.
\end{cor}

For the proof of Theorem~\ref{t.main} our approach is similar to that in \cite{SS1}: one constructs an auxiliary real analytic 
hypersurface $M$ that contains the 
standard umbrella $\Sigma$. The hypersurface $M$ is singular at the origin, but it is smooth and strictly pseudoconvex at all
other points. Then one considers the so-called \it{characteristic foliation} on $\phi(\Sig) \setminus \{0\}$ with respect to $\phi(M)$.
It turns out that certain topological configurations of the phase portrait of the foliation guarantee local polynomial convexity 
of $\phi(\Sig)$ at the origin. Direct computations yield a system of ODEs that determines the phase portrait, however, the system is degenerate, and standard tools from dynamical systems cannot be directly applied. In \cite{SS1} the authors used the theory 
of normal forms of Bruno~\cite{B} and a result of  Dumortier~\cite{D} to determine the phase portrait of the characteristic 
foliation. This was generalized to the smooth case in \cite{SS2}. In our approach we use a result of Brunella and Miari~\cite{BM} 
to reduce the problem of determining the phase portrait of $\phi(\Sig)$ to that of the so-called \it{principal} part of the vector field arising from the foliation. Under certain nondegeneracy conditions on the principal part, its phase portrait is topologically equivalent to that of 
the original vector field. The system obtained in \cite{SS1} has degenerate principal part, and therefore, the result in \cite{BM} 
could not be applied in that case. However, a suitable modification of the auxiliary hypersurface $M$, introduced in this paper, gives 
a system with a nondegenerate principal part. Our final calculations of the phase portrait of the principal part also use Bruno's 
normal form theory. 

The proof of the corollary uses local polynomial convexity of the umbrella established in Theorem~\ref{t.main} and the result of 
Anderson, Izzo and Wermer~\cite{AIW}. The proof of Cor.~1 in~\cite{SS1} goes through in our case without any further modifications,
once the local polynomial convexity is established.

Our interest in open Whitney umbrellas originates in the paper of Givental~\cite{Giv}, who showed that any compact real surface $S$,
orientable or not, admits a so-called Lagrangian inclusion, a map $F: S \to \mathbb C^2$, which is a local Lagrangian embedding except a 
finite number of singularities that are either double points or Whitney umbrellas. It is well-known (see, e.g., \cite{A} or \cite{NS}) that 
certain surfaces do not admit a Lagrangian inclusion $F$ without umbrellas, and so open Whitney umbrellas appear to be intricately related to the topology of the 
surfaces. The study of convexity properties near Whitney umbrellas is an instrumental part in this investigation. In particular, combining 
Theorem~\ref{t.main} with the results in~\cite{SS2} we conclude that \it{any Lagrangian inclusion is locally polynomially convex at every
point}.

%%%%%%%%%%%%%%%%%%%%%%%%%%%%%%%%%%%%%%%%%%
\section{Reduction to a Dynamical System} \label{s.red2ds}

In this section we review how the problem of local polynomial convexity near a Whitney umbrella can be reduced to the 
computation of the phase portrait of a certain dynamical system, a method that was introduced in \cite{SS1}. In fact, 
the procedure works without modifications for a somewhat more general type of isolated singularities.

\subsection{The characteristic foliation}
Let $\tau : \real{2} \rightarrow \real{4} \cong \compx{2}$, $\tau(0)=0$, be a homeomorphism onto its image, smooth except at the origin, 
and such that $S = \tau(\real{2})$ is a totally real surface in $\compx{2}$ with an isolated singular point at the origin. Suppose 
$S$ is embedded in a real hypersurface $M$ in $\compx{2}$. We define a field of lines determined at every $p \in S \setminus \{0\}$ by 
\begin{equation*}
L_p = T_pS \cap H_pM,
\end{equation*}
where $H_pM = T_pM \cap JT_pM$ is the complex tangent space of $M$ at $p$ and $J$ is the standard complex structure on $\compx{2}$. The foliation defined by the integral curves corresponding to this field is called the \it{characteristic foliation} of $S$ (with respect to $M$).

Let us also suppose that $M$ is defined as the zero locus of a function $\rho : \compx{2} \rightarrow \real{}$, smooth and strictly plurisubharmonic  near the origin,
\begin{equation*}
M = M(\rho) = \{(z,w) \in \compx{2} : \rho(z,w) = 0 \}, \hrs{5} \nabla \rho |_{M\setminus\{0\}} \neq 0,
\end{equation*}
and let
\begin{equation*} \label{eq:10}
\Omega(\rho) = \{ (z,w) \in \compx{2} : \rho(z,w) < 0 \}.
\end{equation*}

Recall that the \it{polynomially convex hull} $\hat{K}$ of a compact set $K \subset \compx{2}$ is defined as 
\begin{equation*}
\hat{K} = \{ z \in \compx{n} : \abs{P(z)} \leq \norm{P}_K, \text{ for every holomorphic polynomial } P \}.
\end{equation*}
$K$ is called \it{polynomially convex} if $K = \hat{K}$. Its \it{essential hull} $\ess{K}$ is defined by
$\ess{K} = \close{\hat{K} \setminus K}$, and its \it{trace} $\trc{K}$ by $\trc{K} = \ess{K} \cap K$. We note that
\begin{equation} \label{eq:13}
\ess{K} \subseteq \wh{\trc{K}}.
\end{equation}
Indeed, a local maximum principle due to Rossi \cite{Ros, St} states that if $K$ is a compact set in $\compx{n}$, $E$ is a compact subset of $\hat{K}$ and $U$ is an open subset of $\compx{n}$ that contains $E$, then for all $f \in \mathcal{O}(U)$, $\norm{f}_E = \norm{f}_{(E \cap K)\cup \partial E}$, where the boundary of $E$ is taken with respect to $\hat{K}$. Now, by choosing $E = \ess{K}$ and $U = \compx{2}$, we obtain (\ref{eq:13}).

Since $\tau$ is continuous, the set $S = \tau(\real{2})$ is connected. Let $\varepsilon >0$ be such that $\rho$ is strictly plurisubharmonic 
in $\oball{0}{\varepsilon}$. By a classical result (see, for example, \cite{H, St}), the polynomially convex hull of $\close{S \cap \oball{0}{\varepsilon}}$ agrees with its psh-hull. Hence, the polynomial hull of the set $\close{S \cap \oball{0}{\varepsilon}}$ is contained in 
$\close{\Omega(\rho) \cap \oball{0}{\varepsilon}}$. Let $X$ be the connected component of $S \cap \close{\oball{0}{\varepsilon}}$ 
containing the origin. Then $X \setminus \{0\}$ is a smooth compact real surface embedded in $\partial \Omega(\rho)$. The following key
proposition is essentially due to Duval \cite{Duv} (see also J\"oricke \cite{Jor}).

\begin{prop} \label{prp:1} 
$\trc{X}$ cannot intersect a leaf of the characteristic foliation at a totally real point of $X$ without crossing it.
\end{prop}

The original proof of Duval can be easily adapted to our situation. It is an application of Oka's characterization of polynomially convex subsets of $\compx{n}$. Oka's family of algebraic curves 
can be constructed from the leaves of the characteristic foliation, and because $\Omega$ is strictly pseudoconvex, it suffices to ensure that
the family leaves $\Omega$. See~\cite{SS1} for details.

The last step in reducing the problem to a dynamical system is provided by the following result. Recall that a \it{rectifiable arc} is the homeomorphic image of an interval under a Lipschitz continuous map.

\begin{prop} \label{prp:2}
Suppose that there exist two rectifiable arcs $\gamma_1$, $\gamma_2$ in $X$ such that
\begin{enumerate}[(i)]
\item $\gamma_1 \cap \gamma_2 = \{0\}$;
\item $\gamma_j$ are smooth at all points except, possibly, at the origin;
\item For any compact subset $K \subset X$ not contained in $\gamma_1 \cup \gamma_2$, there exists a leaf $\gamma$ of the characteristic foliation of $S$ such that $K \cap \gamma \neq \emptyset$ but $K$ does not meet both sides of $\gamma$.
\end{enumerate} 
Then, $X$ is polynomially convex.
\end{prop}
\begin{proof}
It follows from Proposition \ref{prp:1} that $\trc{X} \subseteq \gamma_1 \cup \gamma_2$ and from (\ref{eq:13}) that $\ess{X} \subseteq \wh{\gamma_1 \cup \gamma_2}$. A rectifiable arc is polynomially convex \cite[Corollary 3.1.2]{St}. Moreover, by \cite[Theorem~3.1.1]{St}, 
 if $Y$ is a compact polynomially convex subset of $\compx{n}$ and $\Gamma$ is a compact connected set of finite length, then $(\wh{Y\cup\Gamma})\setminus (Y\cup\Gamma)$ is either empty or it contains a complex purely one-dimensional analytic subvariety of the complement $\compx{2} \setminus (Y\cup\Gamma)$. By taking $Y$ and $\Gamma$ to be the arcs $\gamma_1, \gamma_2$, it can be shown by following  the same rationale as in \cite[~Corollary~2]{SS1}, that the union of the two arcs cannot bound a complex one-dimensional variety. Therefore, $\wh{\gamma_1 \cup \gamma_2} = \gamma_1 \cup \gamma_2 \subset X$, so $\ess{X} \subset X$. Since $\wh{X}\setminus X \subseteq \ess{X}\setminus X = \emptyset$, it follows that $X$ is polynomially convex.
\end{proof}

Our next goal is to find a suitable hypersurface containing the open Whitney umbrella, such that the properties of Proposition \ref{prp:2} are satisfied.

\subsection{The characteristic foliation of the open Whitney umbrella} \label{s.redWU}
We identify $\real{4}_{(x,u,y,v)}$ with $\compx{2}_{(z, w)}$ for computational purposes.
If $I_2$ is the $2 \times 2$ identity matrix, we denote by
%\begin{equation*} \label{eq:2}
%\Omega = \begin{pmatrix}
%  0 & I_2 \\
%  -I_2 & 0
%\end{pmatrix}
%\end{equation*}
%the $4 \times 4$ matrix giving the standard symplectic structure on $\compx{2}$ and by
\begin{equation*} \label{eq:3}
J = \begin{pmatrix}
  0 & -I_2 \\
  I_2 & 0
\end{pmatrix}
\end{equation*}
the matrix defining the standard complex structure on $\compx{2}$. Let $\phi : \compx{2} \rightarrow \compx{2}$ be a local symplectomorphism which, without loss of generality, is assumed to preserve the origin. Let the Jacobian matrix of $\phi$ at $0$ be
\begin{equation*} \label{eq:4}
D\phi(0) = \begin{pmatrix}
A & B \\
C & D
\end{pmatrix},
\end{equation*}
where $A,B,C,D$ are the $2\times 2$ block components given by the partial derivatives of $\phi$. Since $\phi$ is symplectic, 
we have
\begin{equation} \label{eq:5}
A^tD-C^tB=I_2, \hrs{5} A^tC = C^tA, \hrs{5} D^tB = B^tD.
\end{equation}
Let $\psi : \real{4} \rightarrow \real{4}$ be the linear transformation given by the matrix
\begin{equation*} \label{eg:8}
\Psi = \begin{pmatrix}
D^t & -B^t\\
B^t & D^t
\end{pmatrix}.
\end{equation*}
Since $\Psi J = J\Psi$, the map $\psi$ complex linear. We now show that $\Psi$ is invertible. From (\ref{eq:5}) we get
\begin{equation} \label{eq:100}
D(\psi \circ \phi)(0) =
\begin{pmatrix}
I_2 & 0\\
E & G
\end{pmatrix}, \hrs{10} E = (e_{ij}), e_{ij} \in \compx{},
\end{equation}
where
\begin{equation} \label{eq:110}
G =(g_{ij}) = B^tB + D^tD.
\end{equation}
Since $D\psi(0)$ is symplectic, $\det D\phi(0) = 1$, and so $\det G = \det \Psi$. We claim that
\begin{equation} \label{eq:11_1}
\det G = g_{11}g_{22}-g_{12}^2 > 0.
\end{equation}
Indeed, let
$B =(b_{jk})$, and $D =(d_{jk})$. A straightforward computation gives
\begin{equation*} \label{eq:102}
\begin{split}
\det G &= (b_{11}b_{22}-b_{12}b_{21})^2 + (b_{11}d_{12}-b_{12}d_{11})^2 + (b_{11}d_{22}-b_{12}d_{21})^2 \\& + (b_{21}d_{12}-b_{22}d_{11})^2 + (b_{21}d_{22}-b_{22}d_{21})^2 + (d_{11}d_{22}-d_{12}d_{21})^2,
\end{split} 
\end{equation*}
which is obviously nonnegative. If $\det G = 0$, then, for $j=1,2$, the following hold
\begin{equation*} \label{eq:103}
%\begin{split}
(b_{j2} = 0) \Rightarrow (b_{j1}=0),\hrs{10} (d_{j2} = 0) \Rightarrow (d_{j1}=0).
%\end{split}
\end{equation*}
On the other hand, if any two or more of $b_{12}, b_{22}, d_{12}, d_{22}$ do not equal $0$, then the corresponding ratios $\displaystyle \frac{b_{11}}{b_{12}}, \frac{b_{21}}{b_{22}}, \frac{d_{11}}{d_{12}}, \frac{d_{21}}{d_{22}}$ are equal, e.g., if $b_{12} \neq 0$, $b_{22} \neq 0$, $d_{12} \neq 0$, and $d_{22} \neq 0$, then
\begin{equation*} \label{eq:104}
\frac{b_{11}}{b_{12}} = \frac{b_{21}}{b_{22}} = \frac{d_{11}}{d_{12}}= \frac{d_{21}}{d_{22}} = \lambda \in \real{}. 
\end{equation*}
It is not difficult to see that all possible combinations lead to $D\phi(0)$ either having two identically zero columns in the vertical $B|D$ block, or one column being a $\lambda$ multiple of another. In both scenarios $\det D\phi(0) = 0$, which is a contradiction. It follows then, that $\det G > 0$, which proves that $\Psi$ is nonsingular. Furthermore, \eqref{eq:110} and \eqref{eq:11_1} imply that $ g_{11} > 0, \hrs{5} g_{22} > 0$. 

Now, let
\begin{equation*} \label{eq:9}
\Sigma' = (\psi \circ \phi)(\Sigma),
\end{equation*}
which by construction is a totally real surface with an isolated singular point at the origin. We consider the following auxiliary hypersurface which contains $\Sigma$,
\begin{equation*} \label{eq:15}
M = M(\rho) = \{ (z,w) \in \compx{2} : \rho(z,w) := x^2 - yv^2 + \frac{9}{4}u^2-y^3 + C(xy-\frac{3}{2}uv) = 0 \}, \ \  C>0.
\end{equation*}
A direct computation shows that for any $C>0$, the gradient $\nabla \rho$  does not vanish in some punctured neighbourhood of the origin. Now, put
\begin{equation*}
M' = (\psi \circ \phi)(M) = M'(\rho'), \hrs{5} \rho' := \rho \circ (\psi \circ \phi)^{-1} .
\end{equation*}
It follows that $M'$ is also smooth in some punctured neighbourhood of the origin. Clearly, $\varphi(\Sigma)$ is locally polynomially convex at the origin if and only if $(\psi \circ \varphi)(\Sigma)$ is. We next show that, for some $C>0$, $M'$ is strictly pseudoconvex near the origin. Let $(x',u',y',v')$ be the coordinates in the target space of $\psi \circ \phi$ and let 
\begin{equation*}
(D(\psi \circ \phi)(0))^{-1} = 
\begin{pmatrix}
I_2 & 0\\
E' & G'
\end{pmatrix}, \hrs{10} E'=(e'_{ij}), G'=(g'_{ij}), \hrs{5} e_{ij}, g_{ij} \in \compx{}.
\end{equation*}
The formal Taylor expansion of $(\psi \circ \phi)^{-1}$ is given by
\begin{equation*}
\begin{split}
(\psi \circ \phi)^{-1}(x',u',y',v') = &\bigg( x' + \sigma^1, u'+\sigma^2, e'_{11}x'+e'_{12}u'+g'_{11}y'+g'_{12}v' + \sigma^3, \\& e'_{21}x'+e'_{22}u'+g'_{12}y'+g'_{22}v' + \sigma^4  \bigg) ,
\end{split}
\end{equation*}
where
\begin{equation*}
\sigma^i = \sum_{j+k+l+m \geq 2} h^i_{jklm} x'^ju'^ky'^lv'^m, \hrs{5} h^i_{jklm} \in \compx{}, \hrs{5} i \in \{1,2,3,4\}.
\end{equation*}
Then,
\begin{equation*}
\begin{split}
\rho'(x',u',y',v') &= (x' + \sigma^1)^2 \\&-(e'_{11}x'+e'_{12}u'+g'_{11}y'+g'_{12}v' + \sigma^3)(e'_{21}x'+e'_{22}u'+g'_{12}y'+g'_{22}v' + \sigma^4)^2 \\&+\frac{9}{4} (u'+\sigma^2)^2 - (e'_{11}x'+e'_{12}u'+g'_{11}y'+g'_{12}v' + \sigma^3)^3 \\& +C(x'+\sigma^1)(e'_{11}x' + e'_{12}u'+g'_{11}y'+g'_{12}v'+\sigma^3) \\& - \frac{3C}{2}(u'+\sigma^2)(e'_{21}x'+e'_{22}u'+g'_{12}y'+g'_{22}v'+\sigma^4).
\end{split}
\end{equation*}

A direct computation gives the Levi form of $\rho'$,
\begin{equation*} \label{eq:15_1}
L_{\rho'} =
\begin{pmatrix}
\displaystyle 2+2Ce'_{11} & \displaystyle Ce'_{12}-\frac{3}{2}Ce'_{21} + \frac{5i}{2}Cg'_{12} \\
\\
\displaystyle Ce'_{12}-\frac{3}{2}Ce'_{21} - \frac{5i}{2}Cg'_{12} & \displaystyle \frac{9}{2} - 3Ce'_{22} \\
\end{pmatrix}.
\end{equation*}
From this it is clear that for $C$ sufficiently small the Levi form is strictly positive-definite. This implies that $\rho'$ is strictly plurisubharmonic near the origin, hence $M'$ is strictly pseudoconvex in some punctured neighbourhood of the origin. Note that the constant $C$ depends on the symplectomorphism $\phi$.

We will show that $S = \Sigma'$ and $M = M'(\rho')$ satisfy the conditions of Proposition \ref{prp:2}. For this, in Section \ref{s.calcul}, we compute the the dynamical system describing the characteristic foliation of $\Sigma'$, and in Section \ref{s.red2QH} we describe the method of reduction to the principal part of a vector field due to Brunella and Miari \cite{BM}. We use this in Section \ref{s.phase} to determine the phase portrait of the characteristic foliation.

%%%%%%%%%%%%%%%%%%%%%%%%%%%%%%%%%%%%%%%%%%
\section{Calculation of the System} \label{s.calcul}
In this section we compute the relevant low order terms of the pullback to the parameterizing plane $\real{2}_{(t,s)}$ of the dynamic system that determines the characteristic foliation of $\Sigma '$. We introduce the following notation for the components of the gradient of $\rho '$,
\begin{equation*}
\nabla \rho' = \left( R_x(t,s), R_u(t,s), R_y(t,s), R_v(t,s) \right) ,
\end{equation*}
and we also set
\begin{equation*}
\sigma^i_x = \frac{\partial \sigma^i}{\partial x'}, \hrs{10} \sigma^i_u = \frac{\partial \sigma^i}{\partial u'}, \hrs{10} \sigma^i_y = \frac{\partial \sigma^i}{\partial y'}, \hrs{10} \sigma^i_v = \frac{\partial \sigma^i}{\partial v'}, \hrs{5} i \in \{1,2,3,4\}. 
\end{equation*}
A straightforward computation gives the Jacobian matrix of $(\psi \circ \phi)^{-1}$ at the origin,
\begin{equation} \label{eq:20}
\begin{split}
D(\psi \circ \phi)^{-1}(0) = 
\begin{pmatrix}
I_2 & 0\\
E' & G'
\end{pmatrix} = 
\begin{pmatrix}
I_2 & 0 \\
\\
-G^{-1}E & G^{-1} \\
\end{pmatrix}.
\end{split}
\end{equation}
The characteristic foliation of $\Sigma'$ is determined at every $p \in \Sigma' \setminus \{0\}$ by
\begin{equation*} \label{eq:130}
L_p\Sigma' = T_p\Sigma' \cap H_pM', \hrs{15} H_pM' = T_pM' \cap J(T_pM').
\end{equation*}
It follows that
\begin{equation*}
\langle JX_p, \nabla \rho' \rangle = 0, \hrs{5} \text{ for all } X_p \in L_p\Sigma', \ p \in \Sigma'.
\end{equation*}
We thus obtain a  smooth vector field $X \in T\Sigma'$, given by 
\begin{equation} \label{eq:13_1}
X = \alpha\frac{\partial f}{\partial t} + \beta\frac{\partial f}{\partial s},
\end{equation}
where $f : \real{2} \rightarrow \real{4}$ is defined as 
\begin{equation*}
f = \psi \circ \phi \circ \pi,
\end{equation*}
and $\alpha, \beta$ are smooth functions on $\real{2}$, satisfying $X_{p=f(t,s)} \in L_{p=f(t,s)}\Sigma', \hrs{5}$ for $p\neq0$. 
Consequently, we can choose
\begin{equation} \label{eq:135}
\alpha(t,s) = \langle J\frac{\partial f}{\partial s}, \nabla \rho' \rangle, \hrs{15} \beta(t,s) = -\langle J\frac{\partial f}{\partial t}, \nabla \rho' \rangle.
\end{equation}
We conclude that the characteristic foliation of $\Sigma'$ is defined by the following system of ODE's
\begin{equation} \label{eq:140}
\begin{cases}
\dot{t} = \alpha(t,s)\\
\dot{s} = \beta(t,s).
\end{cases}
\end{equation}
Writing 
\begin{equation*}
f(t,s) = (f_1(t,s), f_2(t,s), f_3(t,s), f_4(t,s)) ,
\end{equation*}
and using \eqref{eq:100} and \eqref{eq:1}, we can express each $f_i$ as a formal power series in $(t,s)$:
\begin{equation} \label{eq:140a}
\begin{split}
f_1(t,s) &= ts + f_{02}^1s^2 + f_{12}^1ts^2 + f_{21}^1t^2s + f_{03}^1s^3 + \sum_{j+k \geq 4} f_{jk}^1t^js^k,\\
f_2(t,s) &= \frac{2}{3}t^3 + f_{02}^2s^2 + f_{12}^2ts^2 + f_{21}t^2s + f_{03}s^3 + \sum_{j+k \geq 4} f_{jk}^2t^js^k,\\
f_3(t,s) &= g_{12}s + g_{11}t^2 + e_{11}ts + f_{02}^3s^2 + \frac{2e_{12}}{3}t^3 + f_{12}^3ts^2 + f_{21}^3t^2s + f_{03}^3s^3 + \sum_{j+k \geq 4} f_{jk}^3t^js^k,\\
f_4(t,s) &= g_{22}s + g_{12}t^2 + e_{21}ts + f_{02}^4s^2 + \frac{2e_{22}}{3}t^3 + f_{12}^4ts^2 + f_{21}^4t^2s + f_{03}^4s^3 + \sum_{j+k \geq 4} f_{jk}^4t^js^k,
\end{split}
\end{equation}
From the above identities, putting $\displaystyle X_t = \frac{\partial f}{\partial t}$, $\displaystyle X_s = \frac{\partial f}{\partial s}$, we get
\begin{equation} \label{eq:141}
X_t = \begin{pmatrix}
s+2f_{21}^1ts+f_{12}^1s^2\\
\\
2t^2 + 2f_{21}^2ts+f_{12}^2s^2\\
\\
2g_{11}t + e_{11}s + 2e_{12}t^2 + 2f_{21}^3ts + f_{12}^3s^2\\
\\
2g_{12}t + e_{21}s + 2e_{22}t^2 + 2f_{21}^4ts + f_{12}^4s^2
\end{pmatrix} + o(\abs{(t,s)}^2),
\end{equation}
and
\begin{equation} \label{eq:142}
X_s = \begin{pmatrix}
t+2f_{02}^1s+f_{21}^1t^2+2f_{12}^1ts+3f_{03}^1s^2\\
\\
2f_{02}^2s+f_{21}^2t^2+2f_{12}^2ts+3f_{03}^2s^2\\
\\
g_{12}+e_{11}t+2f_{02}^3s+f_{21}^3t^2+2f_{12}^3ts+3f_{03}^3s^2\\
\\
g_{22}+e_{21}t+2f_{02}^4s+f_{21}^4t^2+2f_{12}^4ts+3f_{03}^4s^2
\end{pmatrix} + o(\abs{(t,s)}^2).
\end{equation}
It follows from (\ref{eq:135}) that 
\begin{equation} \label{eq:270}
\begin{split}
\alpha(t,s) = -(X_s)_3R_x - (X_s)_4R_u + (X_s)_1R_y + (X_s)_2R_v = \sum_{j,k \geq 0} \alpha_{jk}t^js^k,\\
\beta(t,s) = (X_t)_3R_x + (X_t)_4R_u - (X_t)_1R_y - (X_t)_2R_v = \sum_{j,k \geq 0} \beta_{jk}t^js^k,
\end{split}
\end{equation}
where $(X_t)_i, (X_s)_i$, $i = 1, \dots, 4$, are the components of $X_t$, $X_s$, respectively. By a direct inspection using \eqref{eq:140a}, \eqref{eq:141}, \eqref{eq:142} and \eqref{eq:270}, we find that the terms up to order $3$ in the power expansion of $\alpha(t,s)$ are given by
\begin{equation*} \label{eq:28_1}
\alpha_{01}s = C\left( -g'_{11}g^2_{12} +\frac{1}{2}g'_{12}g_{12}g_{22} +\frac{3}{2} g'_{22}g_{22}^2   \right) s,
\end{equation*}
\begin{equation*} \label{eq:29_1}
\alpha_{20}t^2 = C\left( -g'_{11}g_{11}g_{12} - g'_{12}g^2_{12}  + \frac{3}{2}g'_{12}g_{11}g_{22} + \frac{3}{2}g'_{22}g_{12}g_{22}  \right) t^2,
\end{equation*}
and those of $\beta(t,s)$ by
\begin{equation*} \label{eq:30}
\beta_{11}ts = 2C\left( g'_{11}g_{11}g_{12} + g'_{12}g_{11}g_{22} - \frac{3}{2}g'_{12}g_{12}^2 - \frac{3}{2}g'_{22}g_{22}g_{12}  \right) ts,
\end{equation*}
\begin{equation*} \label{eq:31}
\beta_{30}t^3 = 2C\left( g'_{11}g_{11}^2 -\frac{1}{2}g'_{12}g_{12}g_{11} - \frac{3}{2}g'_{22}g_{12}^2 \right) t^3.
\end{equation*}
Replacing the primed coefficients with their expressions from (\ref{eq:20}) we obtain
\begin{equation} \label{eq:31a}
%\begin{split}
\alpha_{01} = (3C/2)g_{22},\hrs{10}
\alpha_{20} = -Cg_{12},\hrs{10}
\beta_{11} = -3Cg_{12},\hrs{10}
\beta_{30} = 2Cg_{11}.
%\end{split}
\end{equation}
Since $g_{11}, g_{22}$ are positive, it follows that
\begin{equation} \label{eq:31e}
\alpha_{01}>0 \mbox{ and } \beta_{30}>0.
\end{equation}
and, for $g_{12}\neq 0$,
\begin{equation} \label{eq:31f}
g_{12} \alpha_{20} <0 \mbox{ and } g_{12} \beta_{11} <0.
\end{equation}

Combining all of the above, the dynamical system (\ref{eq:140}) defining the characteristic foliation of $\Sigma'$ becomes
\begin{equation} \label{eq:system}
\begin{cases}
\dot{t} = \alpha(t,s) = \frac{3C}{2} g_{22} s - Cg_{12} t^2 + o (\abs{t}^2 + \abs{s} ),\\ \dot{s} = \beta(t,s) = -3C g_{12} ts + 2Cg_{11} t^3 + o(\abs{t}^3 + \abs{ts}).
\end{cases}
\end{equation}

%%%%%%%%%%%%%%%%%%%%%%%%%%%%%%%%%%%%%%%%%%%%%%%%%%%%
\section{Reduction to the Principal Part}\label{s.red2QH}

To prove that the system \eqref{eq:system} defines a characteristic foliation satisfying the conditions of Proposition \ref{prp:2}, we need to determine the topological structure of the vector field $X$ in \eqref{eq:13_1}. Although the linear part of $X$ does not vanish, its eigenvalues do vanish, making the origin a nonelementary isolated singularity of $X$. Therefore, we cannot apply standard results, such as the Hartman-Grobman theorem. Instead, we will make use of a result by Brunella and Miari \cite{BM} which, under certain conditions, reduces the problem to determining the topological class of a truncated vector field.

Let $\Chi$ be a $C^{\infty}$-smooth vector field on $\real{2}_{(x_1,x_2)}$, with an isolated nonelementary singularity at the origin. Its power series expansion at $0$ can be written as
\begin{equation} \label{eq:55}
\Chi(x) = \sum_{j=1,2} \sum_Qf_{jQ}\;x^Q x_j \hrs{2} \frac{\partial}{\partial x_j}, \hrs{5} \text{ where } Q = (q_1,q_2) \in \integ{2}, q_j \geq -1, \text{ and } x^Q = x_1^{q_1} x_2^{q_2}. 
\end{equation} 
We also assume that $f_{1(i,-1)}=f_{2(-1,i)}=0$ for all $i \in \nat{}\cup \{-1\}$ (here $\nat{} = \{0,1,2, \dots \}$). We call the subset of $\real{2}$ defined by
\begin{equation*}
\bf{D} = \{ Q + (1,1) : Q \in \integ{2}, \abs{f_{1Q}} + \abs{f_{2Q}} \neq 0  \}
\end{equation*}
the \it{support} of of the vector field $\Chi$.
The \it{Newton polygon} of $\Chi$ is defined as the convex hull $\Gamma$ of the set
\begin{equation*}
\bigcup_{Q \in \bf{D}} \{Q+P : P \in \real{2}_+  \}, 
\end{equation*} 
where $\real{}_{+} = [0, +\infty)$. It coincides with the intersection of all support half spaces of $\bf{D}$ (see \cite{B}). The boundary of $\Gamma$ consists of edges, which we denote by $\Gamma_j^{(1)}$, and vertices, which we denote by $\Gamma_j^{(0)}$, where $j$ is some enumeration and the upper index denotes the dimension of the object. The union of the compact edges of $\Gamma$, which we denote by $\hat{\Gamma}$, is called the \it{Newton diagram}  of $\Chi$. % (or the \it{open Newton polygon}, in the terminology of \cite{B}).

\begin{defn}
Let $\Chi$ be given as in (\ref{eq:55}). 
\begin{enumerate}[$(i)$]
\item The vector field
\begin{equation*}
\pp{\Chi}(x) = \sum_{j =1,2} \sum_{Q \in \hat{\Gamma}} f_{jQ}\;x^Q x_j  \frac{\partial}{\partial x_j}
\end{equation*}
is called the \it{principal part} of $\Chi$.
\item Let $\Gamma_1^{(1)},  \dots, \Gamma_N^{(1)}$, $N>0$, be all the (compact) edges in the Newton diagram. The vector field
\begin{equation*}
\Chi_k(x) = \sum_{j =1,2} \sum_{Q \in \Gamma_k^{(1)}} f_{Q}\;x^Q x_j  \frac{\partial}{\partial x_j}
\end{equation*}
is called the \it{quasi-homogeneous component} of the principal part $\pp{\Chi}(x)$ relative to $\Gamma_k^{(1)}$, for $k = 1,\dots, N$.
\end{enumerate}
\end{defn}

\begin{defn}
Let $\Chi_1$, $\Chi_2$ be two planar vector fields defined on the open subsets $\Omega_1$ and $\Omega_2$ of $\real{2}$, respectively.   We say that $\Chi_1$ is \it{topologically equivalent to} $\Chi_2$ if there exists a homeomorphism $h : \Omega_1 \rightarrow \Omega_2$ sending orbits of $\Chi_1$ to orbits of $\Chi_2$. More precisely, if $\gamma_1$ is the orbit of $\Chi_1$ passing through $p \in \Omega_1$, then $h(\gamma_1)$ is the orbit of $\Chi_2$ passing through $h(p)$. In this case we say that $\Chi_1$ and $\Chi_2$ belong to the same topological class of vector fields.
\end{defn}

In general, if a $C^{\infty}$-smooth planar vector field does not have characteristic orbits (i.e., orbits approaching the singular point in positive or negative time with a well-defined slope limit), then an isolated singularity is either a centre or a focus, or briefly, a \it{centre-focus}. Following the terminology introduced in \cite{BM}, we say that two vector fields on $\real{2}$, $\Chi_1$ and $\Chi_2$, $\Chi_1(0)=\Chi_2(0) = 0$, are locally topologically equivalent \it{modulo centre-focus} if either one of the following cases apply:
\begin{enumerate}[$(i)$]
\item $\Chi_1$ and $\Chi_2$ have characteristic orbits and are topologically equivalent near the origin, or
\item $\Chi_1$ and $\Chi_2$ are both centre-foci.
\end{enumerate}

Following Brunella and Miari, we say that a  $C^{\infty}$-smooth planar vector field $\Chi$, $\Chi(0) = 0$, has a 
\it{nondegenerate} principal part $\pp{\Chi}$, if none of its quasi-homogeneous components has singularities 
on $(\real{}\setminus \{0\})^2$. The main result of Brunella and Miari is the following:

\vrs{3}

\it{Let $\mathcal{X}$ be a $C^{\infty}$-smooth vector field on $\real{2}$, $\mathcal{X}(0)=0$, with nondegenerate 
principal part  $\pp{\Chi}$, such that $0$ is an isolated singularity of $\pp{\Chi}$. Then $\mathcal{X}$ is locally 
topologically equivalent to $\pp{\Chi}$ modulo centre-focus.}

\begin{figure}[h]
\centering
\includegraphics[width=0.5 \textwidth]{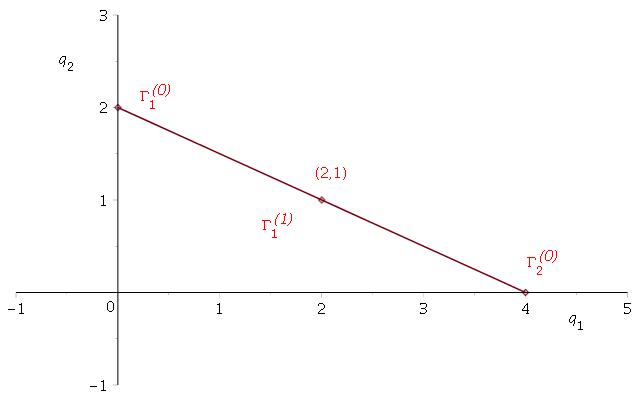}
\caption{The Newton diagram for (\ref{eq:system}).}
\label{fig:1}
\end{figure}

For the system \eqref{eq:system}, the Newton diagram consists of the two vertices $\Gamma^{(0)}_1 = (0, 2)$, $\Gamma^{(0)}_2 = (4,0)$ and the edge $\Gamma^{(1)}_1$ connecting them, see Figure \ref{fig:1}. The principal part of $X$ is given by
\begin{equation}\label{eq:320}
X_{\mathit{\Delta}}(t,s) = \left( \alpha_{01}s + \alpha_{20}t^2 \right)\frac{\partial}{\partial t} + \left( \beta_{11}ts + \beta_{30}t^3 \right)\frac{\partial}{\partial s} .
\end{equation}
Notice that $X_{\mathit{\Delta}}$ also counts for the terms corresponding to the vertex $\Gamma_{3}^{(0)} = (2,1) \in \Gamma^{(1)}_1$.
Clearly, $X_{\mathit{\Delta}}$ has only one quasi-homogeneous component, that is $X_{\mathit{\Delta}}$ itself. We claim that
$X_{\mathit{\Delta}}$ is nondegenerate. Indeed, a singular point $(t,s)$ of $X_{\mathit{\Delta}}$ would satisfy the system
\begin{equation} \label{eq:33}
\begin{cases}
\alpha_{01}s + \alpha_{20}t^2 = 0\\ \beta_{11}ts + \beta_{30}t^3 =0.
\end{cases}
\end{equation}
Note that, if $t=0$, the only solution of \eqref{eq:33} is the origin, hence we can assume $t \neq 0$. Thus, we obtain a linear system in $s$ and $t^2$, that has nonzero solutions if and only if $\alpha_{01}\beta_{30} = \alpha_{20}\beta_{11}$. However, this is impossible, since, by (\ref{eq:31a}) and (\ref{eq:11_1}), we have
\begin{equation} \label{eq:36_1}
%\begin{split}
\alpha_{01}\beta_{30} - \alpha_{20}\beta_{11} = 3C^2g_{11}g_{22} - 3C^2g_{12}^2 = 3C^2(g_{11}g_{22} - g_{12}^2) > 0.
%\end{split}
\end{equation}
This proves that $X_{\mathit{\Delta}}$ is nondegenerate, with one isolated singularity at the origin. Thus, by Brunella and Miari it suffices to compute the phase portrait of $X_{\mathit{\Delta}}$.

%%%%%%%%%%%%%%%%%%%%%%%%%%%%%%%%%%%%%%%%%%%%%%%%%%%%%%%%%
\section{Final Step: the Phase Portrait} \label{s.phase}

%\subsection{Phase portrait of the open umbrella} \label{ssec:phsptr}
Recall that the principal part of the vector field defined by \eqref{eq:system} is given by \eqref{eq:320}, 
and the corresponding ODE system is
\begin{equation} \label{eq:37}
\begin{cases}
\dot{t} = \alpha_{01}s + \alpha_{20}t^2 = t(\alpha_{01}t^{-1}s + \alpha_{20}t) \\ \dot{s} = \beta_{11}ts + \beta_{30}t^3 = s(\beta_{11}t + \beta_{30}t^3s^{-1}) .
\end{cases}
\end{equation}

We determine the phase portrait of  $X_{\mathit{\Delta}}$ near the origin using Bruno's theory of normal forms~\cite{B}. For each element $\Gamma_j^{(d)}$ of the Newton diagram associated with \eqref{eq:37}, there is a corresponding sector  $\mathcal U^{d}_j$ in the phase space $\rl^2_{(t,s)}$, so that together they form a full neighbourhood of the origin (here boundaries of the sectors are not necessarily integral curves). 
In each $\mathcal U^{d}_j$ one brings the system to a normal form by using power transformations (quasihomogeneous blow-ups)
which reduces the problem to the study of elementary singularities of the transformed system. This allows one to determine the behaviour of the orbits in each sector. After that the 
results in each sector are glued together to obtain the overall phase portrait of the system near the origin. In what follows we apply Bruno's 
method to our specific case, referring the reader to \cite{B} or \cite[Section 5]{SS1} for more details on the general method.

We remark that in Bruno's version, the Newton polygon differs from that defined in Section~\ref{s.red2QH} by a translation of $(-1,-1)$,
and so the Newton diagram of $X_{\mathit{\Delta}}$ now consists in two vertices, $\Gamma^{(0)}_1=(-1,1)$, $\Gamma^{(0)}_2=(3,-1)$, and 
one edge $\Gamma^{(1)}$ connecting $\Gamma^{(0)}_1$ and $\Gamma^{(0)}_2$.  By Bruno's classification~\cite[p 138]{B}, the vertices are of Type I, so the integral curves in the sectors
\begin{equation*}
\begin{split}
&\mathcal{U}_1^{(0)}(\varepsilon) = \{(t,s) \in \real{2} : (\abs{t}, \abs{s})^{(1,0)} \leq \varepsilon, \hrs{3}  (\abs{t}, \abs{s})^{(-2,1)} \leq \varepsilon \}, \\& \mathcal{U}_2^{(0)}(\varepsilon) = \{(t,s) \in \real{2} : (\abs{t}, \abs{s})^{(0,1)} \leq \varepsilon, \hrs{3}  (\abs{t}, \abs{s})^{(2,-1)} \leq \varepsilon \},
\end{split}
\end{equation*}
are vertical and horizontal, respectively, in particular, they do not approach the origin.

Next, we analyze the behaviour of the orbits in the sector
\begin{equation*}
\mathcal{U}^{(1)}(\varepsilon) = \{\{(t,s) \in \real{2} : \varepsilon \leq (\abs{t}, \abs{s})^{(-2,1)} \leq \frac{1}{\varepsilon}, \abs{t}, \abs{s} \leq \varepsilon  \}
\end{equation*}
corresponding to the edge $\Gamma^{(1)}$, whose unit directional vector is $R = (-2, 1)$. Following Bruno's method, the vector $R$ leads to the coordinate transformation $y_1 = t, y_2 = t^{-2}s$. After the change of time parameter $d\tau_1 = y_1d\tau$, we obtain the equivalent system
\begin{equation} \label{eq:39}
\begin{cases}
\dot{y}_1 = y_1\left(\alpha_{20} + \alpha_{01}y_2\right),\\ \dot{y}_2 = y_2\left[\beta_{30}y_2^{-1} + (\beta_{11}-2\alpha_{20}) - 2\alpha_{01}y_2 \right].
\end{cases}
\end{equation}
We are interested in the singular points along the $y_2$-axis, i.e., the solutions of the quadratic equation
\begin{equation} \label{eq:40}
-2\alpha_{01}y_2^2 +(\beta_{11}-2\alpha_{20})y_2 + \beta_{30} = 0,
\end{equation}
whose discriminant is
\begin{equation*}
\mathcal{D} = (\beta_{11}-2\alpha_{20})^2 + 8\alpha_{01}\beta_{30}.
\end{equation*} 
By \eqref{eq:31e}, $\mathcal{D}$ is positive, hence (\ref{eq:40}) has two distinct real roots
\begin{equation} \label{eq:41}
y^{\pm} = \frac{\beta_{11} -2\alpha_{20} \pm\sqrt{(\beta_{11}-2\alpha_{20})^2 + 8\alpha_{01}\beta_{30}}}{4\alpha_{01}}.
\end{equation}
We need to analyze the dynamics near each point $(0, y^{\pm})$, and to do so, we translate $y^\pm$ to the origin via he following change of coordinates
\begin{equation*} \label{eq:41_0}
z_1 = y_1, \hrs{5} z_2 = y_2 - y^\pm.
\end{equation*}
As a result, the system \eqref{eq:39} becomes
\begin{equation*} \label{eq:39_1}
\begin{cases}
\dot{z}_1 = z_1\left[(\alpha_{20} + \alpha_{01}y^\pm) + \alpha_{01}z_2\right],\\ \dot{z}_2 = z_2\left[(\beta_{30} + \beta_{11}y^\pm-2\alpha_{20}y^\pm -2\alpha_{01}(y^\pm)^2 )z_2^{-1} + (\beta_{11}-2\alpha_{20} -4\alpha_{01}y^\pm) - 2\alpha_{01}z_2 \right].
\end{cases}
\end{equation*}
This is a system whose linear part does not vanish, and its eigenvalues are given by
\begin{equation} \label{eq:41_1}
%\begin{split}
\lambda_1^{\pm} = \alpha_{20}+\alpha_{01} y^{\pm}, \hrs{10} \lambda_2^{\pm} = \beta_{11} -2\alpha_{20} - 4\alpha_{01}y^{\pm}.
%\end{split}
\end{equation}

\begin{lemma} \label{lem:10}
In the above setting, the following inequalities hold,
%\begin{enumerate}[(i)]
\begin{equation*}
\lambda_1^+ >0, \hrs{10} \lambda_1^- <0, \hrs{10} \lambda_2^+ < 0, \hrs{10} \lambda_2^- > 0.
\end{equation*}
%\end{enumerate}
\end{lemma}

\begin{proof}
Suppose first that $g_{12}=0$. Then \eqref{eq:31a} implies that $\alpha_{20}=\beta_{11}=0$, hence $\displaystyle y^{\pm} = \pm\sqrt{\frac{\beta_{30}}{2\alpha_{01}}}$. The corresponding eigenvalues become
\begin{equation*} \label{eq:41_b}
%\begin{split}
\lambda_1^{\pm} = \pm\alpha_{01} \sqrt{\frac{\beta_{30}}{2\alpha_{01}}}, \hrs{15} \lambda_2^{\pm} = \mp 4\alpha_{01}\sqrt{\frac{\beta_{30}}{2\alpha_{01}}},
%\end{split}
\end{equation*}
and by \eqref{eq:31e}, none of them can equal zero. This proves the lemma in the case $g_{12}=0$ so, for the rest of the proof, we assume $g_{12} \neq 0$.

We next observe that, by substituting (\ref{eq:41}) in (\ref{eq:41_1}), we obtain
\begin{equation*}
\lambda_2^{\pm} = \mp\sqrt{(\beta_{11}-2\alpha_{20})^2 + 8\alpha_{01}\beta_{30}},
\end{equation*}
which, by \eqref{eq:31e}, cannot be zero, hence the last two inequalities of the lemma follow.

Suppose now that
\begin{equation*} \label{eq:48}
\lambda_1^+ = \alpha_{20}+\alpha_{01} y^+ \leq 0.
\end{equation*}
Then, by substituting the expression (\ref{eq:41}) for $y^+$, we get
\begin{equation} \label{eq:50}
\beta_{11} + 2\alpha_{20} + \sqrt{(\beta_{11} - 2\alpha_{20})^2 + 8\alpha_{01}\beta_{30}} \leq 0.
\end{equation}
By \eqref{eq:31f}, if $g_{12} <0$ then $\beta_{11} + 2\alpha_{20} > 0$, hence (\ref{eq:50}) cannot be true. If $g_{12} >0$ then $\beta_{11} + 2\alpha_{20} < 0$, so (\ref{eq:50}) leads to
\begin{equation*} \label{eq:51}
(\beta_{11} - 2\alpha_{20})^2 + 8\alpha_{01}\beta_{30} < (\beta_{11} + 2\alpha_{20})^2,
\end{equation*}
which, after simplifications, becomes
\begin{equation*} \label{eq:52}
\alpha_{01}\beta_{30} -  \alpha_{20}\beta_{11}< 0,
\end{equation*}
hence contradicting (\ref{eq:36_1}). Thus, in both cases, we conclude that $\lambda_1^+ > 0$.

Substituting $y^-$ in the first equation of (\ref{eq:41_1}) with its expression (\ref{eq:41}), we get
\begin{equation} \label{eq:53}
\lambda_1^- = \frac{1}{4}\left( 2\alpha_{20} + \beta_{11} - \sqrt{(\beta_{11}-2\alpha_{20})^2 + 8\alpha_{01}\beta_{30}} \hrs{2} \right).
\end{equation}
If $g_{12}>0$ then $2\alpha_{20} + \beta_{11} <0$, hence by (\ref{eq:53}), $\lambda_1^- <0$. If $g_{12}<0$, then $2\alpha_{20} + \beta_{11} >0$. In this case, suppose $\lambda_1^-\geq 0$. By (\ref{eq:53}), it follows that
\begin{equation*}
2\alpha_{20} + \beta_{11} \geq \sqrt{(\beta_{11} - 2\alpha_{20})^2 + 8\alpha_{01}\beta_{30}},
\end{equation*}
and since $2\alpha_{20} + \beta_{11} >0$,
\begin{equation*}
(2\alpha_{20} + \beta_{11})^2 \geq (\beta_{11} - 2\alpha_{20})^2 + 8\alpha_{01}\beta_{30},
\end{equation*}
which leads to
\begin{equation*}
8(\alpha_{01}\beta_{30} -\alpha_{20}\beta_{11}) \leq 0 .
\end{equation*}
Again, this contradicts (\ref{eq:36_1}), and it follows that $\lambda_1^- <0$, which proves the lemma.
\end{proof}

By Lemma \ref{lem:10}, for both $y^+$ and $y^-$, the corresponding eigenvalues are of opposite signs, hence the phase portrait of system \eqref{eq:39} is a saddle at the origin. It follows that, in $(y_1,y_2)$-coordinates, the $y_2$-axis and the lines $\{y_2=y^+\}$, $\{y_2=y^-\}$ are integral curves. Let $L_1 = \{(y_1,y^+) : y_1>0\}$, $L_2 = \{(y_1,y^+) : y_1<0\}$, $L_3 = \{(y_1,y^-) : y_1>0\}$, $L_4 = \{(y_1,y^-) : y_1<0\}$, $L_5 = \{(0,y_2) :~ y_2>~y^+\}$, $L_6 = \{(0,y_2) : y_2<y^-\}$ and $I = \{(0, y_2) : \min\{y^-, y^+\} < y_2 < \max\{y^-, y^+\}\}$. In the strip $\{(y_1, y_2) : y_1 \in \real{}, \min\{y^-, y^+\} < y_2 < \max\{y^-, y^+\}\}$ of $\real{2}_{(y_1,y_2)}$, the integral curves are asymptotic to $L_1$ and $L_3$ or to $L_2$ and $L_4$, and do not touch $I$. The rest of the orbits are asymptotic to $L_2$, $L_5$ or to $L_5$, $L_1$ or   to $L_6, L_4$ or, finally, to $L_6$ and $L_3$. This means that in the original system there are two integral curves $s=y^{\pm}t^4$ entering the origin while the other integral curves are in the complement of these two curves. Lastly, we observe that for a sufficiently small 
$\varepsilon >0$, the curves $s=y^{\pm}t^2$ enter $\mathcal{U}^{(1)}(\varepsilon)$, which completes the analysis for the edge $\Gamma^{(1)}$ of the Newton diagram.

Gluing the orbits in all three sectors corresponding to  ($\Gamma_1^{(0)}$, $\Gamma_2^{(0)}$, $\Gamma^{(1)}$) , we see that the phase portrait near the origin of the system \eqref{eq:system} is a saddle. By letting $\gamma_1$ and $\gamma_2$ be the curves $s=y^{\pm}t^2$, we conclude that any small enough compact $K$ which is not contained in $\gamma_1 \cup \gamma_2$ will meet one of the orbits of the characteristic foliation at exactly one point, which shows that the conditions of Proposition \ref{prp:2} are met. This completes the proof.

%%%%%%%%%%%%%%% Biblio %%%%%%%%%%%%%%%%%%%%%%%

\end{document}